\documentclass[5p, times]{article}
\usepackage{epstopdf}
\usepackage{amssymb}
\usepackage{amsmath}
\usepackage{graphicx}
\usepackage{multicol}
\usepackage{tikz}
\usetikzlibrary{shapes.misc}
\usetikzlibrary{matrix}
\newtheorem{theorem}{Theorem}[section]

\newtheorem{corollary}[theorem]{Corollary}

\newtheorem{definition}[theorem]{Definition}
\newtheorem{example}[theorem]{Example}

\newtheorem{note}[theorem]{Note}

\newenvironment{proof}[1][Proof]{\noindent\textbf{#1.} }{\ \rule{0.5em}{0.5em}}
\begin{document}
\begin{center}
  \large
COMPACTNESS WITH IDEALS\\[1ex]
\large
Manoranjan Singha$^*$, Sima Roy$^{**}$\\[2ex]
\scriptsize
\textsuperscript{*}\emph{Department of Mathematics, University of North Bengal, Darjeeling-734013, India}

\textsuperscript{**}\emph{Department of Mathematics, Raja Rammohun Roy Mahavidyalaya, Hooghly-712406, India}

\end{center}

\date{}
\hrule
\vspace{0.2cm}
\begin{flushleft}
  \textbf{ABSTRACT}
\end{flushleft}

One of the main obstacles to study compactness in topological spaces via ideals was the definition of ideal convergence of subsequences as in the existing literature according to which subsequence of an ideal convergent sequence may fail to be ideal convergent with respect to same ideal. This obstacle has been get removed in this article and notions of $\mathcal{I}$-compactness as well as $\mathcal{I}^*$-compactness of topological spaces have been introduced and studied to some extent. Involvement of $\mathcal{I}$-nonthin subsequences in the definition of $\mathcal{I}$ and $\mathcal{I}^*$-compactness make them different from compactness even in metric spaces.\\


\vspace{0.2cm}
\hrule
\footnotetext{
$^*$E-mail: manoranjan$\_$singha@rediffmail.com
\vspace{0.1cm}

$^{**}$E-mail: rs$\_$sima@nbu.ac.in
}

\begin{center}
  \section{Introduction}
\end{center}

The year 1951 saw the raise of statistical convergence by H. Fast\cite{Fast} which was applied to study integrability of certain functions and related summability methods in the year 1959 for the first time by  I. J. Schoenberg\cite{Schoenberg}. During last four decades of the 19th century many mathematicians like J. S. Connor\cite{Connor}, T. Salat\cite{Salatb}, J. Cincura\cite{Cincura}, M. Ma\v{c}aj\cite{Macaj}, G. Di Maio\cite{Maio}, Marek Balcerzak \cite{Balcerzak}, P. Das\cite{Das1}, B. K. Lahiri\cite{Lahiria}, K. Demirci\cite{Demirci}, etc. explored and generalized that concept in various directions. The most generalized version of such convergences is ideal convergence which plays the main role in this article.

Let's begin with some basic definitions and results.

For any non-empty set $X$, a family $\mathcal{I}\subset 2^X$ is called an ideal if $(1)$ $\emptyset \in \mathcal{I}$, $(2)$ $A,B\in \mathcal{I}$ implies $A\cup B\in \mathcal{I}$, and $(3)$ $A\in \mathcal{I},B\subset A$ implies $B\in \mathcal{I}$ \cite{Kuratowski}. An ideal $\mathcal{I}$ is called non-trivial if $\mathcal{I}\neq \{\emptyset\}$ and $X\notin \mathcal{I}.$ A non-trivial ideal $\mathcal{I}\subset 2^X$ is called admissible if it contains all the singleton sets \cite{Kuratowski}. Various examples of non-trivial admissible ideals are given in \cite{Kostyrkoa}.

A sequence $(x_n)_{n\in\mathbb{N}}$ in a topological space $X$ is said to be $\mathcal{I}$-convergent to $\xi\in X$ $(\xi=\mathcal{I}-\lim_{n\to\infty} x_n)$ if and only if for any open set $U$ containing $\xi$,  $\{n\in\mathbb{N}: x_n\notin U\}\in\mathcal{I}$. The element $\xi$ is called the $\mathcal{I}$-limit of the sequence $(x_n)_{n\in\mathbb{N}}$ \cite{Lahirib}.
If $\mathcal{I}$ is an admissible ideal, then the usual convergence in $X$ implies $\mathcal{I}$-convergence in $X$ \cite{Kuratowski}.
Another concept of convergence (called $\mathcal{I}^*$-convergence) closely related to $\mathcal{I}$-convergence.

A sequence $(x_n)_{n\in\mathbb{N}}$ in a topological space $X$ is said to be $\mathcal{I}^*$-convergent to $\xi\in X$ if and only if there exists a set $M\in \mathcal{F}(\mathcal{I})$ $(i.e. \mathbb{N}\setminus M\in \mathcal{I}),$ $M=\{m_1<m_2<...<m_k<...\}$ such that $\lim_{k\to \infty} x_{m_k}=\xi$ \cite{Lahirib}.

The concept of $\mathcal{I}$-convergence of sequences has been extended from the real number space to a metric space \cite{Kostyrkoa}, to a normed linear space \cite{Salatb}, to a finite dimensional space \cite{Pehlivan} even to a topological space \cite{Lahirib}.

In this article $\mathcal{I}$ is a nontrivial admissible ideal on $\mathbb{N},$ unless otherwise stated.

\section{Main Results}

 As stated $\mathcal{I}$-nonthin subsequences play main role in this article, let's begin with few words about nonthinness of subsequences introduced by J. A. Fridy \cite{Fridy}  using density of index sets of the subsequences and define nonthin subsequence with respect to an ideal and go on. As in \cite{Fridy} a subsequence $(x_n)_{n\in K}$ of  a sequence $x=(x_n)_{n\in\mathbb{N}}$ of real numbers is called a thin subsequence if $K$ has density zero otherwise it is known as nonthin.
 For simplicity of writing as well as reading, let's call that a sequence is a mapping whose domain is a cofinal subset of $\mathbb{N}$. Let $x=(x_n)_{n\in L}$ be a sequence in a topological space $X$ and $M$ be a cofinal subset of $L$. Then call $(x_n)_{n\in M}$ a subsequence of $x=(x_n)_{n\in L}$.

 \begin{definition}
 A sequence $x=(x_n)_{n\in M}$ in a topological space $X$ is called $\mathcal{I}$-thin, where $\mathcal{I}$ is a nontrivial admissible ideal on $\mathbb{N}$ if $M\in \mathcal{I}$ ; otherwise it is called $\mathcal{I}$-nonthin.
 \end{definition}
\begin{note}
Let $\mathcal{I}$ be a nontrivial admissible ideal on $\mathbb{N},$ $\mathcal{I}/_M=$ $\{A\cap M; A\in \mathcal{I}\}$ is an ideal on $M.$ $\mathcal{I}/_M$ is nontrivial if $M\notin \mathcal{I}$.
\end{note}

\begin{definition}
Let $(X,\tau)$ be a topological space. For $A\subset X$ and $x\in X$, the $\mathcal{I}$-closure of $A$ is denoted by $\overline{A}^\mathcal{I}=\{x\in X;$ there exists an $\mathcal{I}$-nonthin sequence $(x_n)_{n\in L}$ in $A$ that $\mathcal{I}/_L$-converges to $x\}$  and the $\mathcal{I}^*$-closure of $A$ is denoted by $\overline{A}^{\mathcal{I}^*}=\{x\in X;$ there exists an $\mathcal{I}$-nonthin sequence $(x_n)_{n\in L}$ in $A$ that ${(\mathcal{I}/_L)}^*$-converges to $x\}$.
\end{definition}

$\blacklozenge$ \textbf{Properties of $\mathcal{I}$-closure and $\mathcal{I}^*$-closure}
\begin{itemize}
\item[I.] $\overline{\emptyset}^\mathcal{I}=\emptyset.$
\item[II.] For $a\in A,$ the constant sequence $(a,a,a,\ldots)~\mathcal{I}$-converges to $a,$ this implies $a\in\overline{A}^\mathcal{I}.$ Hence $A\subset\overline{A}^\mathcal{I}.$
\item[III.] Let $A,B$ be subsets of $X$. It's easy to see that $\overline{A}^\mathcal{I}\cup\overline{B}^\mathcal{I}\subset\overline{A\cup B}^\mathcal{I}.$ Now let $x\in {\overline{A\cup B}}^\mathcal{I},$ there exists an $\mathcal{I}$-nonthin sequence $(x_n)_{n\in L}$ in $A\cup B$ such that $x_n\to_{\mathcal{I}/_L}x$. Consider the restrictions $y:\{n\in L:x_n\in A\}=L_1\rightarrow A,~y_n=x_n;~z:\{n\in L:x_n\in B\}=L_2\rightarrow B,~z_n=x_n.$ Since $(x_n)_{n\in L}$ is an $\mathcal{I}$-nonthin sequence, at least one of $y,z$ is an $\mathcal{I}$-nonthin sequence. Without loss of generality assume that $y$ is an $\mathcal{I}$-nonthin sequence. Let $U$ be an open neighbourhood of $x.$ Then $\{n\in L:x_n\not\in U\}\in {\mathcal{I}/_L}.$ Since $\{n\in L_1:y_n\not\in U\}\subset\{n\in L:x_n\not\in U\},$ it follows that $\{n\in L_1:y_n\not\in U\}\in {\mathcal{I}/_{L_1}}$, which implies $x\in\bar{A}^\mathcal{I}.$ Hence $\overline{A\cup B}^\mathcal{I}=\overline{A}^\mathcal{I}\cup\overline{B}^\mathcal{I}.$
\end{itemize}
Similarly, the operator $A\to\overline{A}^{\mathcal{I}^*}$ fixes the empty set, is expansive and commutes with finite unions.

\begin{note}{\rm\cite{Arkhangelskii}} Since sequential closure operator $A\mapsto\bar{A}$ is not idempotent, the $\mathcal{I}$-closure operator is also not idempotent in general.
Let $A$ be the subset of continuous functions in the space $X=\mathbb{R}^\mathbb{R}$ of all real-valued functions on $\mathbb{R}$ with the topology of pointwise convergence and $\mathcal{I}=$ the collection $\mathcal{I}_0$ of all finite subsets of $\mathbb{N}$. Then ${\overline{A}^\mathcal{I}}=B_1$ is the set of all functions of first Baire class on $\mathbb{R}$ and $\overline{\overline{A}^\mathcal{I}}^\mathcal{I}=B_2$ is the set of all functions of second Baire class on $\mathbb{R}$. Since $B_1\neq B_2$, the $\mathcal{I}$-closure operator is not idempotent in general.
\end{note}

\begin{definition}
 The operator $A\mapsto\overline{A}^\mathcal{I}$ induces a topology on $X,$  denoted by $\tau_\mathcal{I}$ and a subset $A$ of $X$ is closed in $\tau_\mathcal{I}$ if and only if $A=\overline{A}^\mathcal{I}.$ In a similar manner, the operator $A\mapsto\overline{A}^{\mathcal{I}^*}$ induces a topology on $X$ too, which is denoted by $\tau_{\mathcal{I}^*}.$
\end{definition}
\begin{theorem}\label{clequalityinfircou}
Let $X$ be a topological space. For any set $A\subset X,$ $A\subset\overline{A}^{\mathcal{I}^*}\subset\overline{A}^\mathcal{I}\subset\bar{A}.$\\
In addition if $X$ is first countable, $\bar{A}=\overline{A}^\mathcal{I}=\overline{A}^\mathcal{I^*}.$
\end{theorem}
\begin{proof}Let $a\in\overline{A}^\mathcal{I}$, there exists an $\mathcal{I}$-nonthin sequence $(a_n)_{n\in L}$ in $A$ such that $a_n\to_{\mathcal{I}/_L} a,~a\in X$. Let $U$ be any open set containing $a$, then $\{n\in\mathbb{N}:a_n\in U\}\in\mathcal{F}({\mathcal{I}/_L})$. Since $\emptyset\not\in\mathcal{F}({\mathcal{I}/_L}),$ there exists a natural number $m$ such that  $m\in\{n\in\mathbb{N}:a_n\in U\}$. Then $a_m\in U\cap A$, which implies $a\in\bar{A}$ and so
$\overline{A}^\mathcal{I}\subset\bar{A}.$ In fact, $A\subset\overline{A}^{\mathcal{I}^*}\subset\overline{A}^\mathcal{I}\subset\bar{A}.$\\
Let $x\in \bar{A}.$ Since $X$ is first countable, there exists a sequence $(x_n)$ in $A$ such that $x_n\longrightarrow x.$ As $\mathcal{I}$ is an admissible ideal,  $x_n\longrightarrow_\mathcal{I} x.$ Thus $x\in \overline{A}^\mathcal{I}.$
\end{proof}

 \begin{theorem}\label{topequalityinfircou}
 Let ${\rm(X,\tau)}$ be a topological space. Then both $\tau_\mathcal{I}$ and $\tau_{\mathcal{I}^*}$ are finer than $\tau.$\\
  In addition if $X$ is first countable, $\tau=\tau_\mathcal{I}=\tau_\mathcal{I^*}.$
 \end{theorem}
 \begin{proof}Let $A\subset X$ be closed in $\tau.$ Then $A=\bar{A}.$ Since $A\subset{\overline{A}}^\mathcal{I}$ and  $A\subset\overline{A}^{\mathcal{I}^*}\subset\overline{A}^\mathcal{I}\subset\bar{A},$ implies $\overline{A}^\mathcal{I}= A,$ which shows that $A$ is closed in $\tau_\mathcal{I}.$ Similarly it follows that $A$ is closed in $\tau_{\mathcal{I}^*}$ as well.\\
 Let $A\subset X$ be closed in $\tau_\mathcal{I}.$ Since $X$ is first countable, using Theorem \ref{clequalityinfircou} $A=\bar{A}$ implies that $A$ is closed in $\tau.$
\end{proof}
\begin{note}
If $(X,\tau)$ be a topological space. Then $\tau_\mathcal{I^*}$ is finer than $\tau_\mathcal{I}$.
\end{note}

\begin{note} A first countable  space $(X,\tau)$ is Frechet compact if and only if every infinite subset of $X$ has a limit point in $(X,\tau_\mathcal{I})$, for some admissible ideal $\mathcal{I}$.
\end{note}
\begin{note}
\textup{A consequence of Theorem \ref{topequalityinfircou} is that if $(X,\tau)$ is $T_2$ then $(X,\tau_\mathcal{I})$ and $(X,\tau_{\mathcal{I}^*})$ are also.
}
\end{note}
The following example shows that $\tau_\mathcal{I}$ is strictly finer than $\tau.$
\begin{example}
\textup{Consider the space $X=[1,\omega_1]$, where $\omega_1$ is the first uncountable ordinal and $\tau$ be the order topology on $X$. Let $(x_n)_{n\in L}$ be any $\mathcal{I}$-nonthin sequence in $X\setminus \{\omega_1\}.$ Consider the ordinal $\beta=\underset{n\in L}\cup x_n$, $\beta$ is countable. Now $(\beta,\omega_1]$ is a neighbourhood of $\omega_1$ containing no elements of $(x_n)_{n\in L}$. So $\{n\in L: x_n \not\in (\beta,\omega_1] \}=L\not\in \mathcal{I}.$ This implies no $\mathcal{I}$-nonthin sequence in $X\setminus\{\omega_1\}$ can $\mathcal{I}/_L$-converge to $\omega_1.$ Hence $[1,\omega_1)$ is closed in $(X,\tau_\mathcal{I})$ but not closed in $(X,\tau).$
}
\end{example}

\begin{definition}
A topological space $X$ is $\mathcal{I}$-compact if any $\mathcal{I}$-nonthin sequence $(x_n)_{n\in K}$ in $X$ has an $\mathcal{I}$-nonthin subsequence $(x_n)_{n\in M}$ that $\mathcal{I}/_M$-converges to some point in $X$.
\end{definition}

\begin{definition}
A topological space $X$ is $\mathcal{I}^*$-compact if any $\mathcal{I}$-nonthin sequence $(x_n)_{n\in K}$ in $X$ has an $\mathcal{I}$-nonthin subsequence $(x_n)_{n\in M}$ that $({\mathcal{I}/_M})^*$-converges to some point in $X$.
\end{definition}

\begin{theorem}\label{clsubcmptcl}
Closed subset of an $\mathcal{I}$-compact space is $\mathcal{I}$-compact.
\end{theorem}

\begin{proof}Let $A$ be any closed subset of an $\mathcal{I}$-compact space $X$. By Theorem \ref{clequalityinfircou}, $A\subset\overline{A}^\mathcal{I}$ $\subset\bar{A}$ which implies that $A=\overline{A}^\mathcal{I}$. Since $X$ is $\mathcal{I}$-compact, every $\mathcal{I}$-nonthin sequence $(x_n)_{n\in K}$ in $A$ has an $\mathcal{I}$-nonthin subsequence $(x_n)_{n\in M}$ that $\mathcal{I}/_M$-converges to $x\in X$ which implies that $x\in \overline{A}^\mathcal{I}=A$, that is $A$ is $\mathcal{I}$-compact.
\end{proof}

Proof of the following Theorem \ref{clsubset} is just a translation of the proof of Theorem \ref{clsubcmptcl} in terms of $\mathcal{I}^*.$
\begin{theorem}\label{clsubset}
A closed subset of an $\mathcal{I}^*$-compact space is $\mathcal{I}^*$-compact.
\end{theorem}
\begin{theorem}
Let $(X,\tau)$ be a topological space, then $\mathcal{I}^*$-compactness $\Rightarrow$ $\mathcal{I}$-compactness.
\end{theorem}
\begin{proof}Let $(x_n)_{n\in L}$ be any $\mathcal{I}$-nonthin sequence in $X$ and since $X$ is $\mathcal{I}^*$-compact, it has an $\mathcal{I}$-nonthin subsequence $(x_n)_{n\in M}$ which is $({\mathcal{I}/_M})^*$-converges to $x.$ So there exists a set $K\in \mathcal{I}/_M$ such that $M\setminus K=\{p_1<p_2<...<p_i<...\}$ with $\lim x_{p_i}=x.$ Now for any open set $U$ containing $x$, there exist $i_0\in \mathbb{N}$ such that $x_{p_i}\in U$ for all $i\geq i_0$. So $\{n\in M; x_n\notin U\}$ $\subset \{p_1,p_2,...p_{i_0}\} \cup K$ $\in \mathcal{I}/_M,$ this implies $(x_n)_{n\in M}$ is $\mathcal{I}/_M$-converges to $x.$
\end{proof}

The following example shows that the reverse implication may not be true.

\begin{example}\label{examp1}
\textup{Let $X=[0,1] $, a subspace of $\mathbb{R}$ with usual topology and $\mathbb{N}=\bigcup_{j=1}^\infty \Delta_j$ be a decomposition of $\mathbb{N}$ such that each $\Delta_j$ is infinite and $\Delta_i\cap\Delta_j=\phi$ for $i\neq j$ (\cite{Kostyrkoa}).
Let $\mathcal{I}=\mathcal{I}_1=$ $\{A\subset \mathbb{N}:$ $A\cap\Delta_i$ is infinite, for finite $i's$ and for other $i's,$ $A\cap\Delta_i$ is finite $\} \cup \mathcal{I}_0$, where $\mathcal{I}_0$ is the class of all finite subsets of $\mathbb{N}$. Clearly $\mathcal{I}$ is a nontrivial admissible ideal. Claim that, $X$ is $\mathcal{I}$-compact. Let $(x_n)_{n\in M}$ be any $\mathcal{I}$-nonthin sequence in $[0,1]$. As in (\cite{Kostyrkob}), let $A_1=\{n\in M; 0\leq x_n\leq\frac{1}{2}\},$ $B_1=\{n\in M; \frac{1}{2}\leq x_n\leq 1\}.$ Then $A_1\cup B_1=M$. Since $M\notin \mathcal{I}$, at least one of them not in $\mathcal{I}$ rename it by $D_1$ and the corresponding interval is $J_1=[a_1,b_1]$ of length $\frac{1}{2}$. So $D_1=\{n\in M ; x_n\in J_1\}\notin \mathcal{I}$. Let $A_2=\{n\in M; x_n\in [a_1, \frac{b_1-a_1}{2}]\}$, $B_2=\{n\in M; x_n\in [\frac{b_1-a_1}{2}, b_1]\}$. Then $A_2\cup B_2=D_1$. Since $D_1\notin \mathcal{I}$, at least one of them not in $\mathcal{I}$ rename it by $D_2$ and the corresponding interval is $J_2=[a_2,b_2]$ of length $\frac{1}{4}$. In the same way we get a sequence of closed intervals $J_1\supset J_2\supset...\supset J_k\supset...,$ where $J_k=[a_k,b_k]$ of length $\frac{1}{2^k}$ and the corresponding sets $D_k=\{n\in M ; x_n\in J_k\}\notin \mathcal{I}$ $(k=1,2,...).$ Therefore there exists a unique $\xi$ such that $\xi\in$${\overset{\infty}{\underset{k=1}\cap}} J_k$. Since $D_1\notin \mathcal{I},$ there exists $K_1\subset D_1$ such that $K_1\cap\Delta_{i_1}$ is infinite say $K_1',$ for some $i_1.$ Similarly
since $D_2\notin \mathcal{I},$ there exists $K_2\subset D_2$ such that $K_2\cap\Delta_{i_2}$ is infinite say $K_2',$ for some $i_2\neq i_1.$
 In the same way for all $n>2$,
since $D_n\notin I,$ there exists $K_n\subset D_n$ such that $K_n\cap\Delta_{i_n}$ is infinite say $K_n',$ for some $i_n\neq i_p, p=1,2,...,n-1$. Let $K={\overset{\infty}{\underset{n=1}\cup}} K_n',$ clearly $K\notin \mathcal{I}.$ We have to show that $(x_n)_{n\in K}\to_{\mathcal{I}/_K} \xi$. Since $\xi\in$${\overset{\infty}{\underset{k=1}\cap}} J_k$, for any $\epsilon>0,$ there exists $m\in\mathbb{N}$ such that $J_m\subset (\xi-\epsilon, \xi+\epsilon).$ Then $\{n\in M ; x_n\notin (\xi-\epsilon, \xi+\epsilon)\}\subset M-D_m$ which implies $ \{n\in K ; x_n\notin (\xi-\epsilon, \xi+\epsilon)\}\subset (M-D_m)\cap K$ $\subset K_1'\cup K_2'\cup...\cup K_{m-1}'$ $\in \mathcal{I}.$ So $\{n\in K ; x_n\notin (\xi-\epsilon, \xi+\epsilon)\}\in \mathcal{I}/_K.$ This implies $(x_n)_{n\in K}$ $\mathcal{I}/_K$-converges to $ \xi$, thus $X$ is $\mathcal{I}$-compact.
Now consider the sequence $(x_n)_{n\in\mathbb{N}}$, where $x_n=\frac{1}{i}, n\in \Delta_i$. Then $(x_n)_{n\in\mathbb{N}}$ have only $\mathcal{I}$-thin convergent subsequences. Thus $X$ is not $\mathcal{I}^*$-compact. Hence, $\mathcal{I}$-compact $\nRightarrow$ $\mathcal{I}^*$-compact.
}
\end{example}

\begin{theorem}\label{seqcom}
Let $(X,\tau)$ be a first countable space, then $\mathcal{I}$-compactness $\Rightarrow$ sequential compactness.
\end{theorem}
\begin{proof}Let $(x_n)_{n\in\mathbb{N}}$ be any sequence in $X.$ Then there exists $M\notin \mathcal{I}$ such that $(x_n)_{n\in M}$ $\mathcal{I}/_M$-converges to $x\in X.$ Since $X$  be first countable, let $\{U_n\}$ be a countable base at $x$ and $V_k= \cap_{i\leq k}U_i$, for all $k\in \mathbb{N}$. Then $\{V_n\}$ be a decreasing sequence of open sets each containing $x$. So $\{n\in M ; x_n\notin V_1\}$ $\in \mathcal{I}/_M.$ As $\mathcal{I}/_M$ is nontrivial, there exists $n_1\in M$ such that $n_1\notin$  $\{n\in M ; x_n\notin V_1\}$ this implies $x_{n_1}\in V_1.$ Further there exists $n_2>n_1$, $n_2\in M $ such that $n_2\notin$ $\{n\in M ; x_n\notin V_2\}$ this implies $x_{n_2}\in V_2$ and so on. Thus $(x_n)_{n\in L}$ is a subsequence of $(x_n)_{n\in\mathbb{N}}$, where $L=\{n_1,n_2,...\}$.
Let $U$ be any open set containing $x$, there exists an open set say, $U_m$ in $\{U_n\}$  such that $U_m\subset U$ and hence $V_n\subset U,$ for all $n\geq m$. This implies $x_n\in U$ for all $n\geq m$ and $n\in L$, so $(x_n)_{n\in L}$ converges to $x$.
\end{proof}
\begin{note}
When $\mathcal{D}$ is a free ultrafilter on $\mathbb{N}$, we have Bernstein's notion of $\mathcal{D}$-compactness and $\mathcal{D}$-limit \cite{Allen}. If $\mathcal{I}$ is the dual maximal ideal to the free ultrafilter $\mathcal{D}$, then every compact space is $\mathcal{I}$-compact. So, in general $\mathcal{I}$-compactness does not imply sequential compactness.
\end{note}
The following example shows that the reverse implication may not be true.
\begin{example}\label{examp2}
Let $X=[0,1]$, a subspace of $\mathbb{R}$ with usual topology. If $K\subset\mathbb{N}$, $K_n=\{p\in K; p\leq n\}$, then the natural density of $K$ is defined by $d(K)=\lim \frac{|K_n|}{n}$, if the limit exists  {\rm(\cite{Halberstem},\cite{Niven})}. Let $\mathcal{I}_d=\{A\subset \mathbb{N} ; d(A)=0\}$ {\rm(\cite{Kostyrkoa})}. Let $x=(x_n)_{n\in\mathbb{N}}$ is a sequence $\{0,0,1,0,\frac{1}{2},1,0,\frac{1}{3},\frac{2}{3},1,...\},$ uniformly distributed in $[0,1]$ {\rm(\cite{Kupers})}. So the density of the index of ${x_k}'s$ in any subinterval of length $d$ is $d$ itself. If $(x_n)_{n\in K}$ is a subsequence of $x$ that $\mathcal{I}/_K$-converges to $\xi$ then, $K\in \mathcal{I}$. For, as in  {\rm(\cite{Fridy})} let $\epsilon>0$ be given and for each n, $K_n=\{p\in K; p\leq n\}$ then
$$|K_n|= |\{p\in K_n : |x_p-\xi|<\epsilon\}|+ |\{p\in K_n : |x_p-\xi|\geq\epsilon\}|$$
$$\leq 2\epsilon n + O(1) $$.
This implies $d(K)\leq 2\epsilon$ and since $\epsilon$ arbitrary, $d(K)=0$. So, compactness $\nRightarrow$ $\mathcal{I}$-compactness.
\end{example}

\begin{note} The Example \ref{examp1} and Example \ref{examp2} show that even in metric spaces compactness, $\mathcal{I}$-compactness and $\mathcal{I}^*$-compactness are different. See Figure \ref{fig}.
\end{note}
\begin{figure}[ht]
\begin{tikzpicture}
\node[] at (-1,0.1) {$\mathcal{I}^*-$Compactness};
\draw[thick,->] (0.5,0) -- (2,0);
\draw[thick,<-] (0.5,0.2) -- (2,0.2);
\node[] at (1.25,0.2) {\small /};
\node[] at (3.5,0.1) {$\mathcal{I}-$Compactness};
\node[] at (1.25,0.6) {$\{0,1\}^{\omega},\mathcal{I}_1$};
\draw[thick,->] (-1,-0.3) -- (-1,-1.8);
\draw[thick,<-] (-1.2,-0.3) -- (-1.2,-1.8);
\node[] at (-1.2,-1) {/};
\node[] at (-2,-1) {$[0,1],\mathcal{I}_1$};
\node[] at (-1.1,-2.1) {Compactness};
\draw[thick,<-] (0.1,-2) -- (3.5,-0.2);
\draw [thick,->] (0.2,-2.2) -- (3.6,-0.4);
\node[] at (1.9,-1.3) {/};
\node[] at (2.7,-1.5) {$[0,1],\mathcal{I}_d$};
\end{tikzpicture}
\caption{Relation among compactness, $\mathcal{I}$-compactness and $\mathcal{I}^*$-compactness in metric spaces }
\label{fig}
\end{figure}
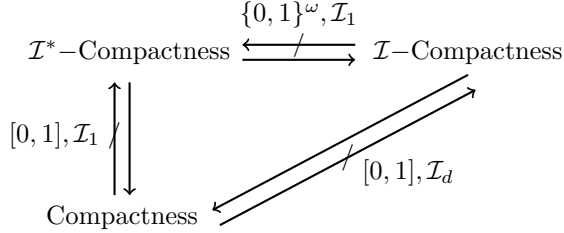

\begin{note}
If $(X,d)$ be an $\mathcal{I}$-compact {\rm(or an $\mathcal{I}^*$-compact)} metric space, then every open cover in $X$ has a Lebesgue number and  the underlying metric space is totally bounded.
\end{note}

\begin{theorem}\label{conimcom}
Continuous image of an $\mathcal{I}$-compact space $(\mathcal{I}^*$-compact space$)$ is $\mathcal{I}$-compact $($resp. $\mathcal{I}^*$-compact$)$.
\end{theorem}

\begin{proof}
The proof is omitted.
\end{proof}
\begin{theorem}\label{prod}
A sequence $(x_n)_{n\in\mathbb{N}}$ in product space $X=\prod\limits_{\lambda\in\Lambda} X_\lambda$ is $\mathcal{I}$-converges to $x\in X$ if and only if $\pi_\lambda (x_n) \to_\mathcal{I}$ $\pi_\lambda (x), \lambda\in\Lambda$.
\end{theorem}

\begin{proof}
Since projection maps are continuous and continuous maps preserve $\mathcal{I}$-convergence, necessary part is trivial.
Let, $x\in \prod\limits_{\lambda\in\Lambda} X_\lambda$ and each $\pi_\lambda (x_n) \to_\mathcal{I}$ $\pi_\lambda (x)$. For any open set $U$ containing $x$, there exists $l_1,l_2,...,l_k\in \Lambda$ and open sets $U_i\subset X_{l_i}$, $i=1,2,...,k$ such that $x\in \pi_{l_1}^{-1}(U_1)\cap\pi_{l_2}^{-1}(U_2)\cap...\cap\pi_{l_k}^{-1}(U_k)$ $\subset U$. This implies $\pi_{l_i}(x)\in U_i, i=1,2,...,k$. Since each $\pi_\lambda (x_n) \to_\mathcal{I}$ $\pi_\lambda (x)$, for open sets $U_i$ containing $\pi_{l_i}(x)$, the set $\{n\in\mathbb{N} ; \pi_{l_i}(x_n)\in U_i\}\in \mathcal{F}(\mathcal{I})$, $i=1,2,...,k$. Again $\{n\in\mathbb{N} ; x_n\in \pi_{l_1}^{-1}(U_1)\}\cap\{n\in\mathbb{N} ; x_n\in \pi_{l_2}^{-1}(U_2)\}\cap...\cap\{n\in\mathbb{N} ; x_n\in \pi_{l_k}^{-1}(U_k)\}$ $\subset\{ n\in\mathbb{N} ; x_n\in U\}$ $\in \mathcal{F}(\mathcal{I})$.
\end{proof}

\begin{theorem}\label{product}
A sequence $(x_n)_{n\in\mathbb{N}}$ in product space $X=\prod\limits_{\lambda\in\Lambda} X_\lambda$ is $\mathcal{I}^*$-converges to $x\in X$ if and only if $\pi_\lambda (x_n) \to_{\mathcal{I}^*}$ $\pi_\lambda (x), \lambda\in\Lambda$.
\end{theorem}
\begin{proof}
The proof is omitted.
\end{proof}

\begin{theorem}\label{finite prod}
A finite product of topological spaces is $\mathcal{I}$-compact $(\mathcal{I}^*$-compact$)$ if and only if every factor is $\mathcal{I}$-compact$($resp. $\mathcal{I}^*$-compact$)$.
\end{theorem}
\begin{proof}Let $X_1$ and $X_2$ be two $\mathcal{I}$-compact spaces and consider two projection maps $\pi_1: X_1 \times X_2\to X_1 $ and $\pi_2: X_1 \times X_2\to X_2 $. Let $(x_n)_{n\in L}$ be any $\mathcal{I}$-nonthin sequence in $X_1 \times X_2$. Then $\pi_1(x_n)={x_n}^1$ $(say)$ and $\pi_2(x_n)={x_n}^2 (say)$  are sequences in $X_1$ and $X_2$ respectively. Since $X_1$ is $\mathcal{I}$-compact, there exists $K\notin \mathcal{I}$ such that $({x_n}^1)_{n\in K} \to_{\mathcal{I}/K} x_1 (say)$ and $K\notin \mathcal{I}$, there exists an $I$-nonthin subsequence $({x_n}^2)_{n\in M}$ of $({x_n}^2)_{n\in K}$ such that $({x_n}^2)_{n\in M} \to_{\mathcal{I}/M} x_2 (say).$ Let $U$ be any open set containing $(x_1,x_2)$, there exists basic open set $B_1 \times B_2\subset U$ containing $(x_1,x_2)$. So $\{k\in M; x_k\notin U\}$ $\subset \{ k\in M; x_k\notin B_1 \times B_2\}$ $\subset \{ k\in M; {x_k}^1\notin B_1\}\cup\{ k\in M; {x_k}^2\notin B_2\}$ $\in \mathcal{I}/_M.$  This shows that $X_1 \times X_2$ is $\mathcal{I}$-compact. Since projection maps are continuous and continuous image of an $\mathcal{I}$-compact space is $\mathcal{I}$-compact, sufficient part is trivial.
The other part of this theorem can be prove in similar way.
\end{proof}

In the following consider two conditions, under one condition $\mathcal{I}$-compactness and $\mathcal{I}^*$-compactness are equivalent in first countable spaces and under another condition compactness and $\mathcal{I}$-compactness are equivalent in metric spaces.
\section{Shrinking Conditions and their applications}
\begin{definition}
An admissible ideal $\mathcal{I}$ is said to satisfy shrinking condition{\rm(A)} if for any sequence of sets $\{A_i\}$ not in $\mathcal{I}$, there exists a sequence of sets $\{B_i\}$ in $\mathcal{I}$ such that each $B_i$ is finite subset of $A_i$ and $B={\overset{\infty}{\underset{i=1}\cup}} B_i\notin \mathcal{I}.$
\end{definition}
The following example is an witness of such ideal.
\begin{example}{\rm(\cite{Kostyrkoa})} \textup{Let $\mathbb{N}={\overset{\infty}{\underset{j=1}\cup}} \Delta_j$ be a decomposition of $\mathbb{N}$ such that each $\Delta_j$ is infinite and $\Delta_i\cap\Delta_j=\phi$ for $i\neq j.$ Let $\mathcal{I}$ denote the class of all $A\subset\mathbb{N}$ which intersect at most finite number of $\Delta_j$s. Then $\mathcal{I}$ is a nontrivial admissible ideal and $\mathcal{I}$ satisfies shrinking condition(A).}
\end{example}

\begin{definition}
An admissible ideal $\mathcal{I}$ is said to satisfy shrinking condition{\rm(B)} {\rm(\cite{Boccuto}Remark 4.15(b))} if for any sequence of sets $\{A_i\}$ not in $\mathcal{I}$, there exists a sequence of sets $\{B_i\}$ in $\mathcal{I}$ such that each $B_i\subset A_i$  and $B={\overset{\infty}{\underset{i=1}\cup}}B_i\notin \mathcal{I}.$
\end{definition}
The following example is an witness of such ideal.

\begin{example}\textup{Let $\mathbb{N}={\overset{\infty}{\underset{j=1}\cup}} \Delta_j$ be a decomposition of $\mathbb{N}$ such that each $\Delta_j$ is infinite and $\Delta_i\cap\Delta_j=\phi$ for $i\neq j.$ Let $\mathcal{I}=$ $\mathcal{I}_0\cup\{A\subset \mathbb{N} :A \cap\Delta_i$ is infinite, for finite $i's$ and for other $i's,$ $A\cap\Delta_i$ is finite $\}$, where $\mathcal{I}_0$ is the class of all finite subsets of $\mathbb{N}$ {\rm(\cite{Kostyrkoa})}.  Then $\mathcal{I}$ is a nontrivial admissible ideal and $\mathcal{I}$ satisfies shrinking condition(B).}
\end{example}

\begin{example}  $\mathcal{I}_d$ satisfies neither shrinking condition$(A)$ nor shrinking condition $(B)$.
\end{example}

\begin{theorem}\label{countprod}
Countable product of $\mathcal{I}$-compact spaces is $\mathcal{I}$-compact if $\mathcal{I}$ satisfy the Shrinking Condition{\rm(B)}.
\end{theorem}

\begin{proof}
Let $(X_i)_{i\in\mathbb{N}}$ be a family of $\mathcal{I}$-compact spaces and $(z_i)_{i\in L}$ is a sequence in $X=\prod_{i=1}^{\infty} X_i$, where $z_i=(x_{j}^{i})$. Then $(x_{1}^{n})_{n\in L}$ is an $\mathcal{I}$-nonthin sequence in $X_1$, as $X_1$ is $\mathcal{I}$-compact, there exists an $\mathcal{I}$-nonthin subsequence  $(x_{1}^{n})_{n\in A_1}$ of  $(x_{1}^{n})_{n\in L}$ that $\mathcal{I}/_{A_1}$-converges to $x_1$(say). Since $X_2$ is $\mathcal{I}$-compact, the $\mathcal{I}$-nonthin sequence $(x_{2}^{n})_{n\in A_1}$ in $X_2$ has an $\mathcal{I}$-nonthin subsequence $(x_{2}^{n})_{n\in A_2}$ that $\mathcal{I}/_{A_2}$-converges to $x_2$. Inductively, the $\mathcal{I}$-nonthin sequence $(x_{1}^{n})_{n\in A_{i-1}}$ in $X_i$ has an $\mathcal{I}$-nonthin subsequence $(x_{1}^{n})_{n\in A_i}$ that $\mathcal{I}/_{A_i}$-converges to $x_i, i\geq 2$. Since $A_1\supset A_2\supset...\supset A_k\supset...$ with each $A_i\notin \mathcal{I}$ and $\mathcal{I}$ satisfies the shrinking condition\rm(B), there exists a sequence of sets $\{B_i\}$ in $\mathcal{I}$ with $B_i$ subset of $A_i$ and $B=\bigcup_{i=1}^{\infty} B_i\notin \mathcal{I}$. Let $C_1=B,$ $C_i=B\setminus \bigcup_{j=1}^{i-1} B_j, i\geq 2$.
Then $(x_{i}^{n})_{n\in C_i}$ is a subsequence of $(x_{i}^{n})_{n\in A_i}$. Since each $\pi_i(z_n)$ has $\mathcal{I}$-nonthin $\mathcal{I}/_{C_i}$-convergent subsequence converges to $\pi_i(x)$, so for any open set $U$ containing $x_i$, $\{n\in {C_i} ; x_{i}^{n}\notin U\}\in \mathcal{I}/_{C_i}$ which implies $\{n\in {C_i} ; x_{i}^{n}\notin U\}\in \mathcal{I}/_B$. For each $i$, $K_i=B\setminus C_i \in \mathcal{I}$ and $\{n\in B ; x_{i}^{n}\notin U\}\subset$ $\{n\in {C_i} ; x_{i}^{n}\notin U\}\cup K_i$ $\in \mathcal{I}/_B$, which implies  $\{n\in B ; x_{i}^{n}\notin U\}\in \mathcal{I}/_B$. Therefore each $(\pi_i(z_n))_{n\in L}$ has $\mathcal{I}$-nonthin $\mathcal{I}/_B$-convergent subsequence converges to $\pi_i(x)$. Using Theorem \ref{prod} the sequence $(z_n)_{n\in L}$ has an $\mathcal{I}$-nonthin subsequence $(z_n)_{n\in B}$ which is $\mathcal{I}/_B$-converges to $x$ and so, $X$ is $\mathcal{I}$-compact.
\end{proof}

\begin{corollary}
Countable product of $\mathcal{I}$-compact spaces is $\mathcal{I}$-compact if $\mathcal{I}$ satisfy the Shrinking Condition{\rm(A)}.
\end{corollary}

The following example shows that countable product of $\mathcal{I}$-compact spaces may not be $\mathcal{I}$-compact.
\begin{example}
Let $X=\prod_{i=1}^{\infty} X_i$, where $X_i=\{0,1\}, i\in\mathbb{N}$ is a topological space with product topology and $\mathcal{I}=\{A\subset\mathbb{N};$ the natural density of $A$ is $0\}$. Clearly $X_i$ is $\mathcal{I}$-compact. Consider a sequence $(z_n)_{n\in\mathbb{N}}$ in $X$, where $(z_n)=(x_m^n)$ defined by
\begin{eqnarray*}
  x_m^n &=& 1, n\in (2m\mathbb{N}+1)\cup(2m\mathbb{N}+2)\cup...\cup(2m\mathbb{N}+m)\cup\{1,2,...,m\} \\
   &=& 0, otherwise
\end{eqnarray*}
since $X_1$ is $\mathcal{I}$-compact, density of the index set of an $\mathcal{I}$-nonthin $\mathcal{I}/_{A_1}$-convergent subsequence $(x_1^n)_{n\in{A_1}}$ of $(x_1^n)_{n\in \mathbb{N}}$ in $X_1$ is less than equal to $1$. so, density of the index set of an $\mathcal{I}$-nonthin $\mathcal{I}/_{A_m}$-convergent subsequence $(x_m^n)_{n\in{A_m}}$ of $(x_m^n)_{n\in{A_{m-1}}}$ in $X_m$ is less than equal to $\frac{1}{2^{m-1}}$. Then $\{A_m\}$ is a decreasing sequence of sets with $d(A_m)\leq \frac{1}{2^{m-1}}$. So for any $\epsilon>0$ and each $m\in \mathbb{N}$, there exists $K_m\in\mathbb{N}$ such that $\frac{\mid{A_m}\cap\{1,2,...,n\}\mid}{n}$ $<\frac{1}{2^{m-1}}+\frac{\epsilon}{3}$, $n\geq K_m$. Let, $\{L_m\}$ be the pairwise disjoint subsets of $\mathbb{N}$ such that $A_{m+1}\subset L_m\subset A_m$, for each m and $d(L_m)=0$. Claim that, $\underset{m\in\mathbb{N}}\cup L_m \in \mathcal{I}$. Since $d(L_m)=0$, for $m\in\{1,2,...,m_0\}$, there exists $n_m\in\mathbb{N}$ such that $\frac{\mid{L_m}\cap\{1,2,...,n\}\mid}{n}<\frac{\epsilon}{3m_0}$, for $n\geq n_m$. Also for that $\epsilon$, there exists $m_0\in\mathbb{N}$ such that $\frac{1}{2^{m_0}}<\frac{\epsilon}{3}$. Let, $l=max\{n_1,n_2,...,n_{m_0},K_{m_0}\}$. Therefore by well known technique, for $n\geq l$, $\frac{\mid({\underset{m\in\mathbb{N}}\cup L_m})\cap\{1,2,...,n\}\mid}{n}$ $={\overset{m_0}{\underset{m=1}\sum}}\frac{\mid L_m\cap\{1,2,...,n\}\mid}{n}$ $+\frac{\mid({\overset{\infty}{\underset{m_0+1}\cup}} L_m)\cap\{1,2,...,n\}\mid}{n}$ $<{\overset{m_0}{\underset{m=1}\sum}}\frac{\epsilon}{3{m_0}}$ $+\frac{\mid A_{m_0+1}\cap\{1,2,...,n\}\mid}{n}<\epsilon$ which implies $d(\underset{m\in\mathbb{N}}\cup L_m )$ $=0$. So, $\underset{m\in\mathbb{N}}\cup L_m \in \mathcal{I}$. If $(z_n)_{n\in\mathbb{N}}$ in $X$ has a subsequence $(z_n)_{n\in K}$ which is $(\mathcal{I}/_K)$-convergent then $K\subset \underset{m\in\mathbb{N}}\cup L_m $ which implies $K\in \mathcal{I}$. Therefore, there does not exist any $\mathcal{I}$-nonthin $\mathcal{I}/_K$-convergent subsequence $(z_n)_{n\in K}$ of $(z_n)_{n\in\mathbb{N}}$. Hence, $X$ is not $\mathcal{I}$-compact.

\end{example}

\begin{theorem}\label{countproda}
Countable product of $\mathcal{I}^*$-compact spaces is $\mathcal{I}^*$-compact if $\mathcal{I}$ satisfy the Shrinking Condition{\rm(A)}.
\end{theorem}
\begin{proof}
The proof is omitted.
\end{proof}

The following example shows that countable product of $\mathcal{I^*}$-compact spaces may not be $\mathcal{I^*}$-compact.

\begin{example}
Let $X=\prod_{i=1}^{\infty} X_i$, where $X_i=\{0,1\}, i\in\mathbb{N}$ is a topological space with product topology and $\mathcal{I}=\{A\subset \mathbb{N}:$ $A\cap\Delta_i$ is infinite, for finite $i's$ and for other $i's,$ $A\cap\Delta_i$ is finite $\} \cup \mathcal{I}_0$, where $\mathcal{I}_0$ is the class of all finite subsets of $\mathbb{N}$ and $\mathbb{N}=\bigcup_{j=1}^\infty \Delta_j$ be a decomposition of $\mathbb{N}$ such that each $\Delta_j$ is infinite and $\Delta_i\cap\Delta_j=\phi$ for $i\neq j$. Then $\mathcal{I}$ is a nontrivial admissible ideal and $\mathcal{I}$ does not satisfy shrinking condition$(A)$. Clearly each $X_i$ is $\mathcal{I}^*$-compact. Consider a sequence $(z_n)_{n\in\mathbb{N}}$ in $X$, where $z_n=(x_i^n)$ defined by
\begin{eqnarray*}
  x_i^n &=& 0, n\in\Delta_1\cup\Delta_2\cup...\Delta_i \\
   &=& 1, otherwise
\end{eqnarray*}

If $(z_n)_{n\in\mathbb{N}}$ in $X$ has an $\mathcal{I}$-nonthin subsequence $(z_n)_{n\in K}$ which is $(\mathcal{I}/_K)^*$-convergent then the fact that, no $\mathcal{I}$-nonthin subsequence of $(x_i^n)_{n\in\mathbb{N}}$ can converge to $0$ ensures that $(z_n)_{n\in K}$ $(\mathcal{I}/_K)^*$-converges to $(1,1,1,...)=\alpha$ \rm(say\rm). Then, there exists a set $M\in \mathcal{F}(\mathcal{I}/_K)$ such that $(z_n)_{n\in M}$ converges to $(1,1,1,...)$. Since $M\in\mathcal{F}(\mathcal{I}/_K)\implies M\notin\mathcal{I}$ that is $M\cap\Delta_i$ is infinite, for infinite $i's$. Let, $M\cap\Delta_i$ is infinite for $i=i_p$ and consider the subsequence $(x_{i_p}^n)_{n\in M}$ of $(x_i^n)_{n\in K}$ in $X_{i_p}$. Since  $M\cap\Delta_{i_p}$ is infinite and
\begin{eqnarray*}
  x_{i_p}^n &=& 0, n\in(\Delta_1\cup\Delta_2\cup...\Delta_{i_p})\cap M \\
   &=& 1, n\in  M-(\Delta_1\cup\Delta_2\cup...\Delta_{i_p})
\end{eqnarray*}
then $\{n\in M; x_{i_p}^n \notin \{1\}\}$ is an infinite set, so $(x_{i_p}^n)_{n\in M}$ cannot converge to 1 at all. Therefore, there does not exist any $\mathcal{I}$-nonthin $(\mathcal{I}/_K)^*$-convergent subsequence $(z_n)_{n\in K}$ of $(z_n)_{n\in\mathbb{N}}$ in $X$ and so, $X$ is not $\mathcal{I}^*$-compact.
\end{example}

\begin{theorem}\label{condition}
 Let $\mathcal{I}$ be an admissible ideal on $\mathbb{N}.$\\
$(i)$ If $\mathcal{I}$ satisfies shrinking condition{\rm(A)} and $(X,\tau)$ is a first countable $\mathcal{I}$-compact space, then $X$ is $\mathcal{I}^*$-compact.\\
$(ii)$ If $\mathcal{I}$ satisfies shrinking condition{\rm(B)} and $(X,d)$ is a compact metric space, then $X$ is $\mathcal{I}$-compact.
\end{theorem}
\begin{proof}$(i)$ Let $(x_n)_{n\in L}$ be any $\mathcal{I}$-nonthin sequence in $X$. Since $X$ is $\mathcal{I}$-compact, there exists an $\mathcal{I}$-nonthin subsequence $(x_n)_{n\in M}$ of $(x_n)_{n\in L}$ which is $\mathcal{I}/_M$ converges to $x\in X$. Since $X$ is first countable, pick a countable base $\{U_n\}$ at $x$. Let $V_n={\overset{n}{\underset{k=1}\cap}} U_k$, then $\{V_n\}$ is a decreasing sequence of open sets each containing $x$. Consider $D_m=\{ n\in M; x_n\in V_m\} \notin \mathcal{I}$. Since $\mathcal{I}$ satisfies shrinking condition{\rm(A)}, there exists a sequence of sets $\{K_i\}$ such that $K_i$ is finite, $K_i\subset D_i$ and $K={\overset{\infty}{\underset{i=1}\cup}} K_i\notin \mathcal{I}$. Claim that, $(x_n)_{n\in K}$ is converges to $x$. Let $U$ be any open set containing $x,$ there exists an open set $U_p \in \{U_n\}$ such that $V_n\subset U_p\subset U , n\geq p.$ So, $\{n\in K; x_n\notin U\}$ $\subset K_1\cup K_2\cup...\cup K_p.$ Let $s=max\{t; t\in K_1\cup K_2\cup...\cup K_p\}.$ Then, $x_n\in U$ for all $n>s, n\in K.$ Hence $X$ is $\mathcal{I}^*$-compact.\\
$(ii)$ Let $(X,d)$ be a compact metric space, so is totally bounded. For $\epsilon=1,$ there exists a finite set $F_1=\{a^1_1,a^1_2,...a^1_{m_1}\}$ such that $X=B_d(a^1_1,1)\cup B_d(a^1_2,1)\cup...\cup B_d(a^1_{m_1},1).$ Let $(x_n)_{n\in M}$ be any $\mathcal{I}$-nonthin sequence in $X$. Consider $A^1_1=\{n\in M; x_n\in \bar B_d(a^1_1,1)\}$, $A^1_2=\{n\in M; x_n\in \bar B_d(a^1_2,1)\}$,..., $A^1_{m_1}=\{n\in M; x_n\in \bar B_d(a^1_{m_1},1)\}$. Then $M=A^1_1\cup A^1_2\cup...\cup A^1_{m_1}$. Since $M\notin \mathcal{I}$, at least one of them not in $\mathcal{I}$ say $A^1$ and the corresponding closed ball say $V^1$ with $diam(V^1)=2$. Since closed ball in a compact metric space is compact, $V^1$ is compact. For $\epsilon=\frac{1}{2}$, there exists a finite set $F_2=\{a^2_1,a^2_2,...a^2_{m_2}\}$ such that $V^1\subset {\underset{a\in F}\cup} B_d(a,\frac{1}{2})$. Consider $A^2_i=\{n\in A^1; x_n\in \bar B_d(a^2_i,\frac{1}{2})\}$, $i=1,2,...,m_2$, so $A^1=A^2_1\cup A^2_2\cup...\cup A^2_{m_2}$. Since $A^1\notin\mathcal{I}$, at least one of them not in $\mathcal{I}$, rename it by $A^2$ and the corresponding closed ball say $U^2$. Thus $U^2\cap V^1=V^2$ is closed in $V^1$ with $diam(V^2)\leq 1$ and hence $V^2$ is closed in $X$. As $X$ is compact, $V^2$ is compact in $X$. Inductively, for $\epsilon=\frac{1}{n}$ we get a sequence of sets $\{A^n\}$ and the corresponding sequence of closed sets $V^1\supset V^2\supset$ ..., such that $diam(V^n)\leq\frac{2}{n}\to 0$ as $n\to \infty.$ Then there exists a unique $\xi$ such that $\xi \in $ ${\overset{\infty}{\underset{k=1}\cap}} V^k.$ Since $\{A^i\}$ decreasing sequence of sets not in $\mathcal{I}$ and $\mathcal{I}$ satisfies shrinking condition{\rm(B)}, there exists a sequence of sets $\{B^i\}$ in $\mathcal{I}$ such that $B^i\subset A^i$ and $B=\underset{i \in\mathbb{N}}\cup B^i\notin \mathcal{I}$. Claim that, $(x_n)_{n\in B}$ is $\mathcal{I}/_B$ converges to $ \xi$. Let $U$ be any open ball containing $\xi,$ there exists $k\in \mathbb{N}$ such that $V^k\subset U.$ Now, $\{n\in M ; x_n\notin U\}\subset \{n\in M ; x_n\notin V^k\} $ which implies $\{n\in B ; x_n\notin U$ $ \}\subset \{n\in M ; x_n\notin V^k\}\cap B$ $\subset B^1\cup B^2\cup...\cup B^K$ $\in \mathcal{I}$ and so, $\{n\in B ; x_n\notin U\}$ $\in \mathcal{I}/_B$. Hence $X$ is $\mathcal{I}$-compact.
\end{proof}\\

\begin{note}
The only condition on $\mathcal{I}$ that will ensure that compactness implies $\mathcal{I}$-compactness is that every $\mathcal{I}$-nonthin set has an $\mathcal{I}$ nonthin subset $M$ so that $\mathcal{I}/_M$ is a maximal ideal. The Stone-Cech compactification of the integers is not $\mathcal{I}$-compact for any ideal that fails to have this property. Also if $\mathcal{I}$-compactness implies $\mathcal{I}^*$-compactness, then every $\mathcal{I}$-nonthin set has an $\mathcal{I}$-nonthin subset $M$ satisfying that $\mathcal{I}/_M$ is the ideal of finite subsets of $M$.
\end{note}

\end{document}